\documentclass[12pt,a4paper,draft]{amsart}
\usepackage[english]{babel}
\usepackage{amssymb}
\usepackage{amscd}
\usepackage[mathscr]{eucal}

\newcommand{\adef}{\begin{defin}}
\newcommand{\zdef}{\end{defin}}
\newcommand{\tsum}{{\textstyle\sum}}
\newtheorem{defin}{Definition}

\newtheorem{thm}{Theorem}[section]
\newtheorem{cor}[thm]{Corollary}
\newtheorem{lem}[thm]{Lemma}

\theoremstyle{definition}

\newcommand{\N}{\mathbb N}

\theoremstyle{remark}

\newcommand\restr[2]{{
		\left.\kern-\nulldelimiterspace 
		#1 
		\right|_{#2} 
}}
\newcommand{\tto}{\longrightarrow}
\allowdisplaybreaks

\numberwithin{equation}{section}
\title{The behaviour of quasi-linear maps on $C(K)$-spaces}
\begin{document}

\author{F\'elix Cabello S\'anchez}
\address{Instituto de Matem\'aticas\\ Universidad de Extremadura\\
Avenida de Elvas\\ 06071-Badajoz\\ Spain} \email{fcabello@unex.es}

\author{Jes\'{u}s M .F. Castillo}
\address{Instituto de Matem\'aticas\\ Universidad de Extremadura\\
Avenida de Elvas\\ 06071-Badajoz\\ Spain} \email{castillo@unex.es}

\author{Alberto Salguero-Alarc\'on}
\address{Departamento de Matem\'aticas\\ Universidad de Extremadura\\
Avenida de Elvas\\ 06071-Badajoz\\ Spain} \email{salgueroalarcon@unex.es}

\thanks{This research has been supported in part by project MTM2016-76958-C2-1-P and Project IB16056 de la Junta de Extremadura. The third author is supported by a grant associated to Project IB16056.}


\maketitle

\begin{abstract} In this paper we combine topological and functional analysis methods to prove that a non-locally trivial quasi-linear map defined on a $C(K)$ must be nontrivial on a subspace isomorphic to $c_0$. We conclude the paper with a few examples showing that the result is optimal, and providing an application to the existence of nontrivial twisted sums of $\ell_1$ and $c_0$.\end{abstract}

\section{Introduction and preliminaries}

This paper treats the behaviour of quasi-linear maps defined on $C(K)$-spaces. More precisely, we will prove that a
quasi-linear map $\Omega: C(K)\to Y$ is either uniformly trivial on finite dimensional subspaces or nontrivial on a copy of $c_0$.
This behaviour of quasi-linear maps on $C(K)$-spaces has consequences about the existence and properties of exact sequences
\begin{equation}\label{exact}
\begin{CD} 0@>>> Y @>>> X @>>>
C(K)@>>>0
\end{CD}\end{equation}

\noindent via the correspondence between exact sequences and quasi-linear maps we describe below. Recall that an \emph{exact sequence} of Banach spaces is a diagram
\begin{equation}\label{exact}
\begin{CD} 0@>>> Y @>>> X @>>>
Z@>>>0
\end{CD}
\end{equation}
formed by Banach spaces and linear continuous operators in which the kernel of each arrow coincides with the image of the
preceding one. The middle space $X$ is usually called a \emph{twisted sum} of $Y$
and $Z$. By the open mapping theorem, $Y$ must be isomorphic to a subspace of $X$ and $Z$ to the quotient $X/Y$. It was the discovery of Kalton \cite{kalt,kaltpeck} that exact sequences of quasi-Banach spaces like \ref{exact} correspond to a certain type of non-linear maps $Z\to Y$ called \emph{quasi-linear maps}, which are homogeneous maps for which there is a constant $Q$ such that
$$\| \Omega(x+y)-\Omega(x)-\Omega(y)\| \leq Q \left(\|x\| +\|y\| \right)$$
for every $x,y\in X$. The least constant $Q$ that satisfies the previous inequality is called the \emph{quasi-linearity constant} of $\Omega$.
Be warned of a surprising fact \cite{CC}: the space $X$ in an exact sequence \ref{exact} is not necessarily a Banach space when $Y$, $Z$ are Banach spaces. Kalton and Roberts \cite{kaltrobe} however proved in a deep theorem that, in combination with \cite{CC}, means that quasi-linear maps $\Omega$ defined on an $\mathcal L_\infty$-space enjoy the additional property that there is a constant $K$ such that for every $n\in\N$ and  every  $x_1,\ldots, x_n\in X$ one has
 \begin{equation}\label{zestimate}
\left\|\Omega\left(\sum_{i=1}^nx_i\right)-\sum_{i=1}^n\Omega(x_i)\right\|
\leq K \; \sum_{i=1}^n\|x_i\|.\end{equation}
and, consequently, exact sequences of quasi-Banach spaces
$$\begin{CD} 0@>>> Y @>>> X @>>>
Z@>>>0\end{CD}$$ in which $Z$ is an $\mathcal L_\infty$-space and $Y$ is a Banach space also have $X$ a Banach space. Thus, we will freely speak about quasi-linear maps $\Omega: C(K) \to Y$  from a space of continuous functions to a Banach space with the understanding that the associated exact sequence $ 0\to Y\to X \to C(K)\to 0$ is a sequence of Banach spaces. We will write $ 0\to Y\to X \to Z \to 0 \equiv \Omega$ to say that the exact sequence and the quasi-linear map correspond one to each other.\medskip

An exact sequence $ 0\to Y\stackrel{\jmath}\to X \to Z \to 0\equiv \Omega$ is said to be \emph{trivial}, or \emph{to split}, if the injection $\jmath$ admits a left inverse; i.e., there is a linear continuous projection $P: X \to Y$ along $\jmath$. In terms of $\Omega$, this means that there exists a linear map $L: Z\to Y$ such that $\|\Omega - L\| = \sup_{\|x\|<1} \|\Omega x -Lx\| <+\infty$. The sequence is said to be $\lambda$-trivial if
$\|\Omega - L\|\leq \lambda$.

\adef We will say that a quasi-linear map $\Omega: X\to Y$ is \emph{locally trivial} if there is a constant $\lambda$ such that if $F$ is  a finite dimensional subspace of $X$ then the restriction $\Omega_{|F}: F \to Y$ is $\lambda$-trivial.\zdef

This notion was introduced by Kalton \cite{kaltloc} who proved that an exact sequence $ 0\to Y\to X \to Z \to 0$ of Banach spaces is locally trivial if and only if its dual exact sequence $ 0\to Z^*\to X^* \to Y^*\to 0$ is trivial. It is also shown in \cite{kaltloc} that a locally trivial sequence $ 0\to Y\to X \to Z \to 0$  in which $Y$ is complemented in its bidual must be trivial.

\section{A dichotomy for  $C(K)$ spaces}

Our aim is to prove the following dichotomy:

\begin{thm}\label{dicho} A quasi-linear map $\Omega: C(K) \to Y$ either is locally trivial or admits a subspace isomorphic to $c_0$ on which
its restriction is not locally trivial.\end{thm}

Before embarking in the proof we need  a technical lemma:

\begin{lem} \label{tecnico} Let $\Omega$ be a quasi-linear map on $X$ with quasi-linearity constant $Q$. Let $x_1,\dots, x_m$ be points so that $\Omega_{|[x_1, \dots, x_m]}$ is $\lambda$-trivial, and let $\delta$ be the Banach-Mazur distance between $[x_1, ..., x_m]$ and $\ell_1^m$. If $y_1,\dots, y_m$ are norm one points satisfying $\|y_i - x_i\|\leq \varepsilon/m\delta$ for all $i=1,\dots, m$, then $\Omega_{|[y_1, \dots, y_m]}$ is $\lambda+\frac{2+\varepsilon}{1-\varepsilon}Q$-trivial. \end{lem}

\begin{proof} Let $L$ be a linear map on $[x_1, \dots, x_m]$ such that $\|\Omega - L\|\leq \lambda$. Set $L'$ the linear map $L'y_i = Lx_i + \Omega(y_i - x_i)$ on $[y_1, \dots, y_m]$. If we call $(*)=\left \|\Omega(\sum \lambda_i y_i) - L'(\sum \lambda_i y_i) \right\|$, then
	
	\begin{eqnarray*}
		(*)&=& \left \|\Omega(\tsum \lambda_i y_i) - L(\tsum \lambda_i x_i) - \tsum  \lambda_i\Omega(y_i - x_i)\right\| \\
		&\leq& \left \|\Omega(\tsum \lambda_i y_i) - \Omega(\tsum \lambda_i x_i) - \tsum \lambda_i \Omega(y_i - x_i) \right\| +  \left\|\Omega(\tsum \lambda_i x _i)- L(\tsum \lambda_i x_i)\right\|\\
		&\leq& \left \|\Omega(\tsum \lambda_i y_i) - \Omega(\tsum \lambda_i x_i) - \Omega( \tsum \lambda_i (y_i - x_i))\right\| + \left\| \Omega( \tsum \lambda_i (y_i - x_i)) - \tsum \lambda_i \Omega(y_i - x_i) \right\|\\&\;& +  \lambda \left\|\tsum \lambda_i x _i\right\|\\
		&\leq& Q\big(\left\|\tsum \lambda_i x_i \right\| + \left \| \tsum \lambda_i (y_i - x_i) \right\| + (m-1)\tsum |\lambda_i|\|y_i - x_i\|\big) + \lambda \left\|\tsum \lambda_i x _i\right\|\\
		&\leq& Q \big( \|\tsum_i \lambda_i y_i\|
		+ m \tsum |\lambda_i|\, \|y_i - x_i\|\big)
		+ \lambda \|\tsum \lambda_i y_i\| \\
		&\leq& Q \big( \|\tsum_i \lambda_i y_i\|
		+ m \tsum |\lambda_i|\, \frac{\varepsilon}{m\delta}\big)
		+ \lambda \|\tsum \lambda_i y_i\| \\
		&\leq & Q\big( \|\tsum\lambda_iy_i\| + \frac{1}{1-\varepsilon}\|\tsum \lambda_iy_i\| \big) + \lambda \|\tsum \lambda_iy_i\| \\
		&\leq& (\lambda+\frac{2+\varepsilon}{1-\varepsilon}Q)	\|\tsum \lambda_iy_i\|. 
\end{eqnarray*}
\end{proof}

We pass to the proof of the theorem.

\begin{proof} We make the proof for the Cantor set $\Delta$. Let $\Omega: C(\Delta)\to Y$ be a quasi-linear map. Split $\Delta = \Delta_- \cup \Delta_+$ in right and left halves $ \Delta_- = \Delta \cap [0,1/3]$ and $ \Delta_+=\Delta \cap [2/3, 1]$. Now, let $J_+ = \{f\in C(\Delta): f_{|\Delta_-}=0\} = C(\Delta_+)$ and $J_- = \{f\in C(\Delta): f_{|\Delta_+}=0\} = C(\Delta_-)$. \\

\textbf{Claim.} \emph{If $\Omega_+ = \Omega_{|J_+} \equiv 0 $ and   $\Omega_- = \Omega_{|J_-} \equiv 0 $ then $\Omega\equiv 0$.}
\begin{proof} If there are linear maps $L_+: C(\Delta_+)\to Y$ and $L_-: C(\Delta_-)\to Y$ so that $\|\Omega_+ - L_+\|\leq M$ and $\|\Omega_- - L_-\|\leq N$, set the linear map
$$L:C(\Delta) \tto Y \quad , \quad Lf=L_+(f1_{\Delta_+})+L_-(f1_{\Delta_-})$$
Since $f = f1_{\Delta_+} + f1_{\Delta_-}$ for every $f\in C(\Delta)$ one has
\begin{equation*} \begin{split}
\|\Omega f - Lf\| & \leq \|\Omega(f)-\Omega(f1_{\Delta_+})-\Omega(f1_{\Delta_-})\|+ \\
	& \quad +\|\Omega(f1_{\Delta_+})-L_+(f1_{\Delta_+})\|+\|\Omega(f1_{\Delta_-})-L_-(f1_{\Delta_-})\| \leq \\
& \leq Q\big(\|f1_{\Delta_+}\|+\|f1_{\Delta_-}\|\big) + M\|f1_{\Delta_+}\| + N\|f1_{\Delta_-}\| \leq  \\
& \leq (M+N+2Q)\|f\| \qedhere \end{split} \end{equation*} \end{proof}

So, if $\Omega$ is non-trivial then either $\Omega_+$ o $\Omega_-$ is non-trivial. For given $\xi\notin\Delta$ we will denote $\Delta_{\xi-} = \Delta \cap [0,\xi]$ and $\Delta_{\xi+} = \Delta \cap [\xi, 1]$. 

\begin{lem}Let $\Omega: C(\Delta) \to Y$ be non-locally trivial. Then there is $\xi\in [0,1]$ so that
\begin{itemize}
\item If $\xi \notin \Delta$ then $\Delta = \Delta_{\xi+} \cup \Delta_{\xi_-}$ with $\Delta_{\xi+} \cap \Delta_{\xi_-}=\emptyset$
and  both $\Omega_{|C(\Delta_{\xi+})}$ and $\Omega_{|C(\Delta_{\xi-})}$ are non-locally trivial.
\item If $\xi \in \Delta$ then $\Delta = \Delta_{\xi+} \cup \Delta_{\xi_-}$ with $\Delta_{\xi+} \cap \Delta_{\xi_-}=\{\xi\}$; and if we call  $J_{\xi+} = \{ f\in \Delta : f_{|\Delta_{\xi_-}}=0\}$ and $J_{\xi-} = \{ f\in \Delta : f_{|\Delta_{\xi_+}}=0\}$ then $\ker \delta_\xi = J_{\xi+} \cup J_{\xi-}$ and both
$\Omega_{|J_{\xi+}}$ and $\Omega_{|J_{\xi-}}$ are non-locally trivial.\end{itemize}\end{lem}
\begin{proof} Split $\Delta = \Delta_+ \cup \Delta_-$. If both $\Omega_+$ and $\Omega_-$ are non-trivial, then we are done: $\xi=1/2$ works and the intersection is empty. Assume, on the contrary, that one of them is trivial, say, $\Omega_+$. Pick $\lambda_1$ so that $\Omega_+$ is $\lambda_1$-trivial but not $\lambda_1-1$-trivial. Iterate the argument. If both $\Omega_{--}$ and $\Omega_{-+}$ (the restrictions of $\Omega$ to $\Delta_{--}$ and $\Delta_{-+}$) are non-trivial, then we are done and $\Delta_-$ splits in two pieces with empty intersection $\Delta_{--}$ and $\Delta_{-+} \cup \Delta_+$, so that $\Omega$ is non-trivial when restricted to them. Otherwise, we continue. If the decomposition method does not stop, let $\xi\in \{-1,1\}^\N$ be the element that represents the non-trivial choices; i.e., the choices of the pieces on which the restriction of $\Omega$ was non-trivial. We can assume without loss of generality that $-\xi$ represents the trivial choices and $\Omega$ is $\lambda_n$-trivial but not $(\lambda_n-1)$-trivial on $\Delta_{\xi(1),\dots, \xi(n-1), -\xi(n)}$, and also non-trivial on $\Delta_{\xi(1),\dots, \xi(n)}$. Now, either $\sup_n \lambda_n  < +\infty$ or $\sup_n \lambda_n = +\infty$. \\

\textbf{Claim.} \emph{If $\sup_n \lambda_n = \lambda< +\infty$ then $\Omega$ is locally trivial}.
\begin{proof} Pick functions $f_1, \dots, f_m$. For $\varepsilon>0$ small enough, each of the functions $f_j$ is at distance $\varepsilon$ of a function of the form $f_j(\xi)\chi +  g_j$ with $g_j $ supported in a union
$$A= \Delta_{-\xi(1)} \cup \Delta_{\xi(1), -\xi(2)} \cup \dots  \cup \Delta_{ \xi(1), \xi(2), \dots, \xi(n-1), -\xi(n)}$$
and $\chi$ is the characteristic function of $\Delta \setminus A$. Since $\Omega$ is $\lambda$-trivial on $C(A)$ and is also 1-trivial on the one-dimensional subspace generated by $\chi$, $\Omega$ is $\mu$-trivial (with $\mu=\lambda+1+2Q$) on the finite-dimensional subspace generated by the functions $f_j(\xi)\chi+g_j$. Therefore, according to lemma \ref{tecnico}, it must  be $(\mu+Q(1+\varepsilon))$-trivial on $[f_1, \dots, f_m]$. \end{proof}

It just remains that $\sup_n \lambda_n = +\infty$. Set the decomposition
\begin{eqnarray*}
\Delta_{\xi-} &=& \Delta_{-\xi(1)}\cup \Delta_{\xi(1), -\xi(2)} \cup \Delta_{\xi(1), \xi(2), -\xi(3)} \cup \dots \\
\Delta_{\xi+} &=& \Delta_{\xi(1), \xi(2)} \cup \Delta_{\xi(1), \xi(2), \xi(3)} \cup \dots \end{eqnarray*}
with $\xi\in \Delta_{\xi+} \cap \Delta_{\xi-} $ and $\Omega$ is non-trivial on both. \\

\par We are ready to conclude the proof of the lemma. The different pieces $\Delta_{\xi(1), \xi(2), \dots, \xi(n-1), -\xi(n)}$  on which $\Omega$ is not $(\lambda_n-1)$-trivial
are disjoint. So when one considers a finite dimensional subspace of $C(\Delta_{\xi(1), \xi(2), \dots,  -\xi(n)})$
where $\Omega$ is not $(\lambda_n-1)$-trivial, it is contained in a subspace $\ell_\infty^{m_n}$ of $C(\Delta_{\xi(1), \xi(2), \dots,  -\xi(n)})$
on which $\Omega$ is not $(\lambda_n-1)$-trivial. The subspaces $\ell_\infty^{m_n}$ are thus disjointly supported, so their closed span yields a copy of $c_0$ where $\Omega$ cannot be locally trivial. \end{proof}

\par Now, if $C(K)$ is isomorphic to $C(\Delta)$, pick $\Omega': C(K)\to Y$ non-locally trivial and $\alpha: C(\Delta)\to C(K)$ an isomorphism. Then $\Omega= \Omega'\alpha$ is not locally trivial, so there is an embedding $j: c_0\to C(\Delta)$ so that $\Omega j$ is not locally trivial. Then $\Omega' \alpha j$ is not locally trivial and $\alpha j: c_0 \to C(K)$ is the embedding one needs.

In general, let  $\Omega: C(K) \to Y$ be  non-locally trivial. Find for each $n=1,2,...$ a finite dimensional subspace $F_n\subset C(K)$ so that $\Omega_{|F_n}$ is not $n$-trivial. Let $X=[\cup_n F_n]$. This is a separable subspace of $C(K)$ on which $\Omega$ is not locally trivial. Let $A$ be the separable subalgebra that $X$ generates in $C(K)$. This $A$ is isometric to $C(H)$ for some compact $H$, necessarily metrizable, and the restriction of $\Omega$ to $C(H)$ cannot be locally trivial as well. If $H$ is countable, then $C(H)=C(\alpha)$ for some countable ordinal $\alpha$ and the result is clear. If $H$ is uncountable then $C(H)=C(\Delta)$ by Milutin's theorem and the result is also clear. \end{proof}

\begin{cor}
A quasi-linear map defined on a $C(K)$-space is locally trivial if and only if every restriction to every copy of $c_0$ is locally trivial.
\end{cor}

\section{Consequences and applications}

First of all, let us show that the result just obtained is the only possible one. We have proved that a quasi-linear map $\Omega: C(K)\to \diamondsuit$ is locally trivial if every restriction $\Omega_{|c_0}$ is locally trivial. However, it is false that a quasi-linear map $\Omega: C(K)\to \diamondsuit$ must be trivial when every restriction $\Omega_{|c_0}$ is trivial, and a simple example like
$$\begin{CD} 0@>>> c_0 @>>> \ell_\infty @>>> \ell_\infty/c_0@>>> 0\end{CD}$$
shows it: by Sobczyk's theorem any restriction to a separable subspace of $\ell_\infty/c_0$ must be trivial. The word ``locally" is therefore not superfluous. It is also false that a locally trivial $\Omega: C(K)\to \diamondsuit$ must have some trivial restriction to some copy of $c_0$, as it is proved by the example in \cite{ccky} of a non-trivial locally trivial sequence
$$\begin{CD} 0@>>> C[0,1] @>>> \diamondsuit @>>> c_0@>>> 0\end{CD}$$
having strictly singular quotient map (from now on called strictly singular sequences). 
The quotient map is not an isomorphism when restricted to any infinite dimensional subspace and consequently the associated quasi-linear map is not trivial when restricted to any infinite dimensional subspace.
To conclude this analysis, observe that strictly singular non locally trivial maps $\Omega: c_0\to Y$ do exist: any exact sequence $0 \to K \to \ell_1 \to c_0 \to 0$ works (it is not locally trivial because a $\mathcal L_\infty$-space cannot be a complemented subspace of a $\mathcal L_1$-space); there also exist strictly singular sequences in which the subspace is complemented in its bidual: 
pick any non-trivial sequence
$ 0\to \ell_2 \to \diamondsuit \to c_0 \to 0$ and apply the method in \cite{castmoresing} to obtain a singular sequence
$$\begin{CD} 0@>>> \ell_\infty(\Gamma, \ell_2) @>>> \clubsuit @>>> c_0@>>> 0\end{CD}$$
Such sequences cannot locally split.

\par Now we pass to applications. Exact sequences
\begin{equation}\label{unocero}\begin{CD} 0@>>> \ell_1 @>>> \spadesuit  @>>> c_0@>>> 0\end{CD}\end{equation}
are somewhat a mistery, even though their existence has been shown in \cite{cabecastuni,ccky}. Both papers contain different although complementary methods to show the existence of non-trivial sequences
\begin{equation}\label{doscero}\begin{CD} 0@>>> L_1(0,1) @>>> \clubsuit  @>>> C[0,1]@>>> 0\end{CD}\end{equation}

Since $L_1$ and $\ell_1$ (resp. $C[0,1]$ and $c_0$) have the same local structure, one can apply the local theory of exact sequences \cite{cabecastuni} to pass from a non-trivial twisting of $L_1$ and $C[0,1]$ to a non-trivial twisting of $\ell_1$ and $c_0$. The final sequence, however, remains invisible since it has emerged via an abstract existence result. Theorem \ref{dicho} yields a much more explicit way to get such
sequences. Let us call $\Omega$ the exact sequence (\ref{doscero}), which is not locally trivial because $L_1$ is complemented in its bidual, and thus there is an injection $\jmath: c_0\to C[0,1]$ so that $\Omega \jmath$, namely the lower sequence in in commutative diagram

\begin{equation}\begin{CD} 0@>>> L_1(0,1) @>>> \clubsuit  @>>> C[0,1]@>>> 0 \equiv \Omega\\
&&@| @AAA @AA{\jmath}A\\
 0@>>> L_1(0,1) @>>> \heartsuit  @>>> c_0@>>> 0 \equiv \Omega \jmath
 \end{CD}\end{equation}
is not locally trivial. This implies that its dual sequence
\begin{equation*}\begin{CD} 0@>>> \ell_1 @>>> \heartsuit^*  @>>> L_1(0,1)^{*}@>>> 0 \equiv \jmath^*\Omega^*\end{CD}\end{equation*}
is not locally trivial. Since $L_1(0,1)^{*}$ is a $C(K)$-space, Theorem \ref{dicho} applies again to obtain an injection $\imath: c_0\to L_1(0,1)^*$ such that the restriction $(\Omega \jmath)^*\imath$, namely, the lower sequence in the commutative diagram
\begin{equation}\begin{CD} 0@>>> \ell_1 @>>> \heartsuit^* @>>> L_1(0,1)^{*}@>>> 0 \equiv (\Omega \jmath)^*\\
&&@| @AAA @AA{\imath}A\\
 0@>>> \ell_1 @>>> \diamondsuit  @>>> c_0@>>> 0 \equiv (\Omega \jmath)^* \imath
 \end{CD}\end{equation}
is not locally trivial. In particular, it is not trivial.\\

One can show that the middle space $\spadesuit$ in an exact sequence (\ref{unocero}) cannot be isomorphic to $\ell_1\oplus c_0$ \cite{castsimo2}. However, it is open to decide whether strictly singular sequences (\ref{unocero}) exist.


\begin{thebibliography}{99}



\bibitem{CC} F. Cabello Sánchez, J.M.F. Castillo, \emph{Duality and twisted sums of Banach spaces,} J. Funct. Anal. 175 (2000) 1–16.

\bibitem{cabecastuni} F. Cabello S\'anchez, J.M.F. Castillo, \emph{Uniform boundedness and twisted sums
of Banach spaces}, Houston J. of Mathematics 30
(2004) 523--536.

\bibitem{ccky} F. Cabello S\'anchez, J.M.F. Castillo, N.J. Kalton, D.T. Yost, \emph{Twisted sums with $C(K)$ spaces,} Trans. Amer.
Math. Soc. 355 (2003) 4523-4541.


\bibitem{castmoresing} J.M.F. Castillo, Y. Moreno, \emph{Strictly singular quasi-linear maps}, Nonlinear Analysis - TMA. 49 (2002)
897-904.


\bibitem{castsimo2} J.M.F. Castillo, M. Sim\~oes, \emph{Positions in $\ell_1$}, Banach J. Math. 9 (2015), no. 4, 395--404.










\bibitem{kalt} N.J. Kalton, \emph{The three-space problem for locally bounded F-spaces}, Compositio Math. 37 (1978) 243-276.

\bibitem{kaltloc} N.J. Kalton, \emph{Locally complemented subspaces and
${\mathcal L}_p$ for $p<1$}, Math. Nachr. 115 (1984),
71-97.

\bibitem{kaltpeck} N.J. Kalton and N.T. Peck, \emph{Twisted sums of sequence spaces and the three-space problem},
Trans. Amer. Math. Soc. 255 (1979) 1-30.

\bibitem{kaltrobe} N.J. Kalton and W. Roberts, \emph{Uniformly exhaustive submeasures and
nearly additive set functions,} Trans. Amer. Math. Soc. 278 (1983)
803-816.

\end{thebibliography}
\end{document}